\documentclass{amsart}

\usepackage{amsmath, mathrsfs, amssymb,amsthm,hyperref,centernot,comment,stmaryrd,enumerate,cancel,mathtools}

\hypersetup{pdfstartview={XYZ null null 1.25}}
\usepackage[all]{xy}

\newcommand{\bv}{\bigvee}
\newcommand{\bw}{\bigwedge}

\newcommand{\cF}{\mathcal F}
\newcommand{\cI}{\mathcal I}

\newcommand{\cU}{\mathcal U}
\newcommand{\cD}{\mathcal D}

\newcommand{\A}{\mathcal A_{\cF,\cI}}
\newcommand{\cA}{\mathcal{A}}

\newcommand{\Pos}{\mathsf{Pos}}
\newcommand{\FUD}{\mathrm{F}(\cU,\cD)}

\newcommand{\rank}{\mathsf{rank}}

\DeclareMathOperator{\R}{R}

\theoremstyle{plain}
\newtheorem{thm}{Theorem}[section]
\newtheorem{prop}[thm]{Proposition}

\newtheorem{lemma}[thm]{Lemma}
\newtheorem{ex}[thm]{Example}
\theoremstyle{definition}
\newtheorem{defn}[thm]{Definition}

\title{Amalgamating poset extensions and generating free lattices}
\author{Rob Egrot}
\date{}
\address{Faculty of ICT, Mahidol University, 999 Phutthamonthon 4 Rd, Salaya, Nakhon Pathom 73170, Thailand}
\email{robert.egr@mahidol.ac.th}
\begin{document}
\begin{abstract}
We investigate connections between the free lattice generated by a poset while preserving certain bounds and the canonical extension of a poset. Explicitly, we describe how the free lattice generated by a poset while preserving certain bounds can be constructed as a colimit of `intermediate structures' as they occur in the construction of a canonical extension of a poset. 
\end{abstract}
\maketitle
\section{Introduction}
A standard technique for constructing the canonical extension of a poset $P$ is to take the sets of all filters and ideals of $P$, and then to define an antitone Galois connection between their powersets using the relation of non-empty intersection. The canonical extension is then the complete lattice of stable sets of filters. This constructive method appeared in \cite{GehHar01} for lattices, and was explicitly applied to construct canonical extensions for posets in \cite{DGP05}, though the technique first appeared in \cite{Tun74}, albeit using different terminology. 

As discussed in \cite[Remark 2.3]{DGP05}, the meanings of the terms `filter' and `ideal' are important here, as definitions that are equivalent for lattices diverge in the more general setting. The effect of varying these definitions on the canonical extension construction is investigated in \cite{Mor14}. 

Going further, it is not necessary to restrict to the sets of \emph{all} filters and ideals, however they are defined, or even to the relation of non-empty intersection. Going down this path leads \cite{MorVanA18} to define canonical extensions relative to a choice of a set of filters and a set of ideals. If we abandon explicit reference to filters, ideals and non-empty intersection altogether, but keep the essential ingredients of the Galois connection construction, we arrive at the generality of $\Delta_1$-completions \cite{GJP13}. This class of completions includes both canonical extensions and MacNeille completions (see e.g. \cite{Mac37,BanBru67}), and is defined to include all completions in which the embedded image of the base poset is \emph{doubly dense} (i.e. every element of the completion is both a join of meets and a meet of joins of subsets from this image). 

The basis of the construction of a $\Delta_1$-completion of a poset $P$ is a triple $(\cF,\cI,\R)$, where $\cF$ and $\cI$ are, respectively, sets of `filters' and `ideals' of $P$ (understood very generally), and $\R\subseteq \cF\times \cI$ is a binary relation. There is a 1-1 correspondence between $\Delta_1$-completions of a poset $P$ and polarities with certain properties (see \cite[Theorem 3.4]{GJP13} for the details, or \cite[Section 7]{Egr20} for a more general result).

If $d:P\to C$ is the $\Delta_1$-completion resulting from  polarity $(\cF,\cI,\R)$, there are natural embeddings of $\cF$ and $\cI$ into $C$. This induces a natural order on $\cF\cup \cI$, producing what is often referred to as the \emph{intermediate structure}. It turns out that the inclusion and reverse-inclusion orders on $\cI$ and $\cF$ respectively agree with the orders induced by $C$ on $\cF\cup \cI$. Thus the intermediate structure is an amalgam of $\cF$ and $\cI$, understood as posets, into a common extension of $P$, using the relation $\R$ as a kind of glue for the two pieces. See \cite[Section 1.3]{vG09} for a discussion of this for a quite general definition of `canonical extension', and \cite{GJP13}, particularly Section 3, for the details in the general setting of $\Delta_1$-completions.  

This intermediate structure can, for the relation of non-empty intersection, be thought of as the `free' way to amalgamate the posets $\cF$ and $\cI$, and comes with a universal property (see \cite[1.4.2]{vG09} and \cite[7.30]{Egr20}). So a canonical extension, for example, is obtained by `freely' combining the chosen $\cF$ and $\cI$, and then completing via the MacNeille completion.

Continuing with the theme of `freeness', given a set $X$, Whitman investigated the \emph{free lattice generated by $X$}, and defined an algorithm for solving the associated word problem \cite{Whi41,Whi42}. Given a poset $P$ we can define the \emph{free lattice generated by $P$ while preserving certain bounds} (see Definition \ref{D:free}). The original construction is due to Dean \cite{Dea64}, and significantly cleaner approach is given by Lakser \cite{Lak12}. Both techniques involve first constructing the `term algebra' of words over $P$, defining a quasiorder over it, and then taking the induced poset to obtain the appropriate free lattice. The advantage of Lakser's approach lies in the definition of the quasiorder. In particular, Lakser replaces Dean's somewhat involved recursive definition with what he calls the \emph{covering condition} \cite[Definition 2]{Lak12}. In this covering condition we see what amounts to the familiar relation of non-empty intersection between filters and ideals.

This raises questions about the relationship between the intermediate structure that appears in the canonical extension construction and the free lattice generated by a poset while preserving certain bounds. Intuitively, we can imagine building this free lattice step by step. First we would add new elements corresponding to joins and meets of subsets of $P$, taking care not to interfere with any of the bounds we wanted to preserve. This would almost certainly not be a lattice, as there would likely be finite subsets of the newly constructed poset without defined joins and meets. Thus we would add more elements corresponding to joins and meets of finite subsets of the poset we constructed in the first stage. This time we would be careful not to interfere with the joins and meets we added the first time. Again, the result of this would likely not be a lattice, but we could keep repeating the process of adding joins and meets indefinitely. The free lattice would be obtained `in the limit' so to speak.

It turns out that this can actually be done. Explicitly, given a poset $P$ we can define a set of `filters' corresponding to the meet structure we want to add, and a set of `ideals' corresponding to the join structure we want to add, and the intermediate structure from the canonical extension construction corresponds to the poset plus added joins and meets. By repeating this process with appropriate further choices, we produce a chain of posets embedding into each other. The desired free lattice can then be constructed by taking the colimit. The details of this are given in Section \ref{S:build}, building on some background results provided in Section \ref{S:background}.  

Finally, in Section \ref{S:approx} we connect the intermediate stages of this construction with a notion of complexity and prove that each stage is, in a sense, a kind of `free' construction (see Theorem \ref{T:rel-free}). To conclude the paper we give an example showing that the `canonical form' theorem for free lattices over sets does not generalize to free lattices over posets preserving certain bounds (Example \ref{E:canon}).

\section{Preliminaries}\label{S:background}
First a little notation. Given a poset $P$ and an element $p\in P$, we define $p^\uparrow = \{p'\in P: p'\geq p\}$, and we define $p^\downarrow$ dually. Given a function $f:X\to Y$ between sets and $Y'\subseteq Y$, we define $f^{-1}(Y') = \{x\in X: f(x)\in Y'\}$. Given $X'\subseteq X$ we define $f[X'] = \{f(x):x\in X'\}$.

\subsection{Free lattices}
To discuss free lattices preserving bounds we first need a way to specify the bounds we wish to preserve. This is done via the following definition.

\begin{defn}\label{D:spec}
Let $P$ be a poset. Let $\cU$ be a subset of $\wp(P)$. Then $\cU$ is a \textbf{join-specification} (of $P$) if it satisfies the following conditions:
\begin{enumerate}[(1)]
\item $\bv S$ exists in $P$ for all $S\in \cU$, and
\item $\{p\}\in \cU$ for all $p\in P$.
\end{enumerate}
A \textbf{meet-specification} is a subset $\cD$ of $\wp(P)$ satisfying (2) and the dual of (1). Given a join-specification $\cU$ we define the \textbf{radius} of $\cU$ to be the smallest cardinal $\sigma$ such that $\sigma> |S|$ for all $S\in\cU$. The radius of a meet-specification is defined dually.
\end{defn}

\begin{defn}[$(\cU,\cD)$-morphism]
Let $f:P\to Q$ be an order-preserving map between posets. Let $\cU$ and $\cD$ be join- and meet-specifications of $P$ respectively. Then $f$ is a $\cU$-morphism if whenever $S\in \cU$ we have $f(\bv S) = \bv f[S]$. Similarly, $f$ is a $\cD$-morphism if whenever $T\in\cD$ we have $f(\bw T) = \bw f[T]$. If $f$ is both a $\cU$-morphism and a $\cD$-morphism then we say it is a $(\cU,\cD)$-morphism. If $f$ is a $\cU$-morphism that is also an order-embedding then we say it is a $\cU$-embedding, and we make similar definitions for $\cD$- and $(\cU,\cD)$-embeddings. 
\end{defn}

\begin{defn}[$\cU$-ideal, $\cD$-filter]\label{D:Uideal}
Let $P$ be a poset, and let $\cU$ and $\cD$ be join- and meet-specifications of $P$ respectively. Then a $\cU$-ideal of $P$ is a downset that is closed under joins from $\cU$, and a $\cD$-filter of $P$ is an upset that is closed under meets from $\cD$. Given a cardinal $\alpha$, we say a $\cU$-ideal or $\cD$-filter of $P$ is $\alpha$-\textbf{generated} if it is the smallest $\cU$-ideal/$\cD$-filter containing $S$ for some $S\subseteq P$ with $|S|<\alpha$. For $\alpha = \omega$ we just say \textbf{finitely generated}. 
\end{defn}

The next lemma proves that inverse images of $(\cU, \cD)$-morphisms produce $\cU$-ideals and $\cD$-filters.

\begin{lemma}\label{L:ideals}
If $h:P\to Q$ is a $(\cU,\cD)$-morphism, then, for all $q\in Q$, \/ $h^{-1}(q^\downarrow)$ is a $\cU$-ideal and $h^{-1}(q^\uparrow)$ is a $\cD$-filter.
\end{lemma}
\begin{proof}
Let $S\in \cU$ and suppose $S\subseteq h^{-1}(q^\downarrow)$. Then $q$ is an upper bound for $h[S]$, and as $h$ is a $(\cU,\cD)$-morphism it follows that $h(\bv S)\leq q$. Since $h^{-1}(q^\downarrow)$ is clearly a downset, it is thus a $\cU$-ideal. The rest follows by duality.
\end{proof}

The next lemma will be important later. The idea is that when $h$ is a $(\cU, \cD)$-morphism and $I$ is a $\cU$-ideal, we find e.g. $\bv h[I]$ by calculating $\bv h[S]$ for a subset $S$ of $I$ that generates $I$. The value of this is that $S$ may be finite, proving that the infinite join $\bv h[I]$ must exist in e.g. a lattice.

\begin{lemma}\label{L:S}
Let $h:P\to Q$ be a $(\cU,\cD)$-morphism, let $I$ be the smallest $\cU$-ideal of $P$ containing $S$, and suppose $\bv h[S]$ exists in $Q$. Then  $\bv h[S] = \bv h[I]$. Similarly, if $F$ is the smallest $\cD$-filter containing $S$ and $\bw h[S]$ exists in $Q$, then $\bw h[S] = \bw h[F]$.
\end{lemma}
\begin{proof}
Let $x\in L$ and suppose that $x$ is an upper bound for $h[S]$. By Lemma \ref{L:ideals}, $h^{-1}(x^\downarrow)$ is a $\cU$-ideal of $P$. As $h^{-1}(x^\downarrow)$ contains $S$, it follows that $I\subseteq h^{-1}(x^\downarrow)$, and thus that $x$ is an upper bound for $h[I]$. In particular, $h(\bv S)$ is an upper bound for $h[I]$. As $\bv h[S] = h(\bv S)$ and $S\subseteq I$, we obtain $\bv h[S] = \bv h[I]$ as required. The rest is dual. 
\end{proof}

\begin{defn}[$\FUD$]\label{D:free}
Let $P$ be a poset, and let $\cU$ and $\cD$ be join- and meet-specifications of $P$ respectively, both with radius at most $\omega$. The \emph{lattice freely generated by $P$ while preserving joins from $\cU$ and meets from $\cD$} is a lattice $\FUD$ such that there is a $(\cU,\cD)$-embedding $e:P\to \FUD$ and such that, whenever $L$ is a lattice and $f:P\to L$ is a $(\cU,\cD)$-morphism, there is a unique lattice homomorphism $u:\FUD\to L$ such that the diagram in Figure \ref{F:free} commutes. 
\end{defn}

\begin{figure}[htbp]
\[\xymatrix{ 
P\ar[r]^{e\phantom{xxxx}}\ar[d]_f & \FUD\ar[dl]^u \\
L
}\] 
\caption{The universal property of $\FUD$}
\label{F:free}
\end{figure}

$\FUD$ always exists, and is unique up to isomorphism fixing $P$ as, demonstrated by the explicit constructions of \cite{Dea64} and \cite{Lak12}.

\subsection{Canonical extensions}\label{S:can-ext}
In \cite{DGP05}, the canonical extension of a poset $P$ was defined in terms of the sets of its up-directed downsets (called \emph{ideals} in that paper), and down-directed upsets (called \emph{filters}). As noted in \cite[Remark 2.3]{DGP05}, this choice of definition for ideal and filter is somewhat arbitrary, and there are others that also agree with the lattice version as used in \cite{GehHar01}. For example, \cite{MorVanA18} defines filters to be upsets closed under existing finite meets, and defines ideals dually. This paper also generalizes the definition of canonical extension by defining it relative to a set $\cF$ of filters and a set $\cI$ of ideals, provided the pair $(\cF,\cI)$ satisfies certain conditions. Thus we can speak of `the canonical extension of $P$ with respect to $(\cF,\cI)$'.

Generalizing further, we can relax the conditions on $\cF$ and $\cI$ to allow the former to be any standard collection of upsets, and the latter to be any standard collection of downsets. Here a \textbf{standard} collection of upsets of $P$ is one that contains all the principal upsets, and the definition for downsets is dual. 

\begin{defn}\label{D:canon} A \textbf{canonical extension of $P$ with respect to $(\cF,\cI)$} is a completion $e:P\to C$ such that the following all hold:
\begin{enumerate}
\item $e$ is \textbf{$(\cF,\cI)$-dense}, by which we mean that given $z\in C$, we have 
\begin{align*}z &= \bv\{ \bw e[F]: F\in\cF \text{ and } \bw e[F]\leq z \}\\&=\bw\{ \bv e[I]: I\in\cI \text{ and } \bv e[I]\geq z \}.\end{align*}  
\item $e$ is \textbf{$(\cF,\cI)$-compact}, by which we mean that whenever $F\in\cF$ and $I\in\cI$, if $\bw e[F]\leq\bv e[I]$ we must have $F\cap I\neq \emptyset$.  
\end{enumerate}  
\end{defn}  

Definition \ref{D:canon} corresponds to that of an $(\cF,\cI)$-completion from \cite[Definition 5.9]{GJP13}, and specializes, after a little fiddling, to the definitions of the canonical extension from \cite[Section 4]{MorVanA18} and \cite[Definition 2.2]{DGP05} by restricting the possible choices of $\cF$ and $\cI$.

Given a poset $P$ and standard sets of upsets and downsets $\cF$ and $\cI$, the canonical extension of $P$ with respect to $(\cF,\cI)$ is unique up to isomorphism, and can be constructed by first amalgamating $\cF$ and $\cI$ (see Definition \ref{D:amal} below), and then taking the MacNeille completion of the resulting poset. See \cite{GJP13}, in particular Theorems 5.10 and 3.4 for proofs applicable to the general setting we are using here. 

\begin{defn}[$\A$, $\pi$, $\tau$, $i$, $f$, $\gamma$]\label{D:amal}
Let $P$ be a poset and let $\cI$ and $\cF$ be standard sets of downsets and upsets of $P$, respectively. Define $\A$ by taking the union $\cF \cup \cI$ and adding the partial order structure induced by the following quasiordering:
\begin{enumerate}
\item For $F_1,F_2\in \cF$, \/ $F_1\leq F_2 \iff F_1\supseteq F_2$.
\item For $I_1,I_2\in \cI$, \/ $I_1\leq I_2 \iff I_1\subseteq I_2$.
\item For $I\in \cI$ and $F\in \cF$:
\begin{enumerate}
\item $F\leq I \iff F\cap I\neq \emptyset$.
\item $I\leq F \iff$ for all $p,q \in P$, if $p\in I \text{ and } q \in F$, then $p\leq q$.
\end{enumerate}
\end{enumerate}
There are maps $\pi:\cI\to\A$ and $\tau:\cF\to\A$ induced by the respective inclusions of $\cI$ and $\cF$ into $\cF\cup \cI$. Define $i:P\to \cI$ and $f:P\to \cF$ by $p\mapsto p^\downarrow$ and $p\mapsto p^\uparrow$ respectively. Define $\gamma:P\to \A$ by $\gamma = \pi\circ i = \pi\circ f$.
\end{defn}

The following lemma collects together some useful properties of the maps from the previous definition.
\begin{lemma}\label{L:gamma}
Let $i$, $f$, $\pi$, $\tau$ and $\gamma$ be as in Definition \ref{D:amal} for some choice of $(\cF,\cI)$. Then:
\begin{enumerate}
\item $\pi$ is a completely join-preserving order-embedding.
\item $\tau$ is a completely meet-preserving order-embedding.
\item $\gamma$ is an order-embedding.
\item If $S\subseteq P$ and $\bv S$ exists in $P$, then 
\[\gamma(\bv S) = \bv \gamma[S] \iff i(\bv S)=\bv i[S].\]
\item If $T \subseteq P$ and $\bw T$ exists in $P$, then 
\[\gamma(\bw T) = \bw \gamma[T] \iff f(\bw T)= \bw f[T].\]
\end{enumerate}
\end{lemma}
\begin{proof}
$\pi$ is obviously an order embedding, as by definition $I_1\leq I_2 \iff I_1\subseteq I_2$. Now, let $X\subseteq \cI$, and suppose $\bv X$ exists in $\cI$. Then $\bv X$ is the smallest element of $\cI$ containing $\bigcup X$. Let $F\in\cF$ with $I\leq F$ for all $I\in X$, and let $q\in F$. Then for all $p\in \bigcup X$ we have $p\leq q$, and so $p^\downarrow \subseteq q^\downarrow$. So $q^\downarrow$ is an upper bound for $X$ in $\cI$, and so $\bv X \subseteq q^\downarrow$. This is true for all $q\in F$, so $\bv X\leq F$, by definition of the order on $\A$. This proves (1), and (2) is dual.

For (3), that $\gamma$ is an order-embedding is immediate as it is the composition of two order-embeddings. For (4), note that $\gamma = \pi\circ i$ and $\pi$ is completely join-preserving. The argument for (5) is dual.
\end{proof}

Composing $\gamma$ with the MacNeille completion $d$ of $\A$ produces $d\circ \gamma$, which is the canonical extension of $P$ with respect to $(\cF,\cI)$.

When $\cI$ and $\cF$ are standard sets of $\cU$-ideals and $\cD$-filters, respectively, the maps $i$, $f$ and $\gamma$ preserve the specified joins and meets, as made precise in the following lemma. 

\begin{lemma}\label{L:gamma2}
Let $\cU$ be a join-specification, and let $S\in \cU$. Then $i(\bv S) = \bv i[S]$ and $\gamma(\bv S) = \bv \gamma[S]$.
Similarly, if $\cD$ is a meet-specification and $T\in \cD$, then $f(\bw T) = \bw f[T]$ and $\gamma(\bw T) = \bw \gamma[T]$. 
\end{lemma}
\begin{proof}
First, $i(\bv S) = (\bv S)^\downarrow$ is a $\cU$-ideal containing $S$. Moreover, any $\cU$-ideal containing $S$ must contain $\bv S$, by virtue of being a $\cU$-ideal. Thus $i(\bv S)$ is the smallest $\cU$-ideal containing $S$, and so is $\bv i[S]$. The argument for $f$ is dual. The claims for $\gamma$ then follow by Lemma \ref{L:gamma}(4).
\end{proof}

\subsection{Directed colimits in the category of posets}\label{S:lims}

Define $\Pos$ to be the category of posets with order-preserving maps. Define $\Pos_e$ to be the category of posets and order-embeddings. We present the following definition, primarily to fix a notation. Details and background can be found in e.g. \cite{MacL98}.

\begin{defn}
If $I$ and $\mathscr C$ are categories, and if $F:I\to \mathscr C$ is a functor, then a \textbf{colimit} for $F$ is a pair $(L, \{f_i:i\in I\})$ such that $L$ is an object of $\mathscr C$, and $f_i$ is a map from $Fi$ to $L$ for all objects $i\in I$ such that:
\begin{enumerate}
\item If $g:i\to j$ is a map in $I$ then $f_i = f_j\circ Fg$.
\item If $C\in\mathscr C$ and for each $i\in I$ there is $h_i: Fi\to C$ such that $h_i = h_j\circ Fg$ for all $i,j\in I$ and all maps $g:i\to j$, then there is a unique map $u:L\to C$ such that the diagram in Figure \ref{F:colim} commutes, for all $i,j,g$.
\end{enumerate}
\end{defn} 

\begin{defn}
A poset is \textbf{directed} if every pair of elements has an upper bound.
\end{defn}

\begin{figure}[htbp]
\[\xymatrix{ 
Fi\ar[dd]_{Fg}\ar[dr]^{f_i}\ar@/^1pc/[drr]^{h_i} \\
& L\ar[r]^u & C\\
Fj\ar[ur]_{f_j}\ar@/_1pc/[urr]_{h_j}
}\] 
\caption{Commuting diagram for a colimit}
\label{F:colim}
\end{figure}

\begin{prop}\label{P:coLim} 
Let $I$ be a directed poset considered as a category. If $i\leq j$ in $I$ we denote the map from $i$ to $j$ in $I$ by $g_{ij}$. Then $\Pos$ and $\Pos_e$ have all colimits of shape $I$, i.e. colimits exist for every functor $F:I\to \Pos$ and $F:I\to \Pos_e$. Moreover, if $F:I\to \Pos$ and $(L,\{f_i:i\in I\})$ is a colimit for $F$, then:
\begin{enumerate}
\item If $F$ can be considered as a functor from $I$ to $\Pos_e$, i.e. if $F(g_{ij})$ is an order-embedding for all $i\leq j\in I$, then a colimit of $F:I\to\Pos_e$ is also a colimit for $F:I\to\Pos$.
\item For all $p\in L$ there is $i\in I$, and $x\in Fi$, with $p = f_i(x)$. 
\item  Let $p,q\in L$, let $i,j\in I$, and suppose $p = f_i(x)$ and $q = f_j(y)$ for some $x\in Fi$ and $y\in Fj$. Then, of the following statements the implications $(a)\implies(b)\implies(c)$ hold. Moreover, if $F$ can be considered as a functor from $I$ to $\Pos_e$ then the statements are all equivalent.
\begin{enumerate}
\item For all $k\geq \{i,j\}$ in $I$ we have $Fg_{ik}(x)\leq Fg_{jk}(y)$ in $Fk$.
\item There is $k\geq \{i,j\}$ in $I$ with $Fg_{ik}(x)\leq Fg_{jk}(y)$ in $Fk$. 
\item $p\leq q$. 
\end{enumerate}
\item $L$ is a lattice if, for all $i,j\in I$, and for all $x\in Fi$ and $y\in Fj$, the following conditions both hold:
\begin{enumerate}
\item There is $k\geq \{i,j\}$ in $I$ and $z\in Fk$ with $g_{in}(x)\vee g_{jn}(y)=f_{kn}(z)$ in $Fn$ for all $n\geq k$.
\item There is $k\geq \{i,j\}$ in $I$ and $z\in Fk$ with $g_{in}(x)\wedge g_{jn}(y)=f_{kn}(z)$ in $Fn$ for all $n\geq k$.
\end{enumerate}
If $F$ can be considered as a functor from $I$ to $\Pos_e$ then the converse (only if) is also true.
\end{enumerate} 
\end{prop}
\begin{proof}
This follows from general model theoretic considerations (see e.g. \cite[Theorems 2.4.5 and 2.4.6]{Hodg93}). Direct proof by construction is also straightforward. $L$ is constructed by first taking the union $\bigcup_I Fi$, and then taking the quotient of this with respect to the quasiordering given by $x\preceq y$ if and only if there is $i$ such that $x,y\in Fi$ and $x\leq y$. The $f_i$ maps are induced by the inclusions into $\bigcup_I Fi$.
\end{proof}

\section{Building free lattices}\label{S:build}
Let $P$ be a poset, and, recalling Definition \ref{D:spec}, let $\cU$ and $\cD$ be join- and meet-specifications of $P$ respectively, both with radius $\omega$. We make definitions as follows:

\begin{itemize}
\item Define $\cA_0= P$.
\item Define $\cU_0$ and $\cD_0$ by $\cU_0=\cU$ and $\cD_0= \cD$.
\item Define $\cI_0$ and $\cF_0$ to be, respectively, the sets of all non-empty finitely generated  $\cU_0$-ideals and $\cD_0$-filters of $P$ (recall Definition \ref{D:Uideal}). Treat these as posets by ordering by inclusion and reverse inclusion respectively.
\item Define $i_0:P\to \cI_0$ and $f_0:P\to \cF_0$ by $i_0:p\mapsto p^\downarrow$, and $f_0:p\mapsto p^\uparrow$. 
\item Define $\cA_1$ to be the amalgam of $\cF_0$ and $\cI_0$ as in Definition \ref{D:amal}.
\item Define $\pi_0:\cI_0\to \cA_1$ and $\tau_0:\cF_0\to\cA_1$ to be the maps induced by the inclusion functions.
\item Define $\gamma_0:P\to \cA_1$ by $\gamma_0 = \pi_0\circ i_0 = \tau_0\circ f_0$.
\end{itemize}
For $n\geq 1$ we make definitions as follows:
\begin{itemize} 
\item Define $\cU_n= \cD_n$ to be the set of non-empty finite subsets of $\gamma_{n-1}[\cA_{n-1}]$.
\item Define $\cI_n$ and $\cF_n$ to be, respectively, the sets of all non-empty finitely generated  $\cU_n$-ideals and $\cD_n$-filters of $\cA_{n}$.
\item Define $i_n:\cA_n\to \cI_n$ and $f_n:\cA_n\to \cF_n$ by $i_n:x\mapsto x^\downarrow$ and $f_n:x\mapsto x^\uparrow$.
\item Define $\cA_n$ to be the amalgam of $\cF_{n-1}$ and $\cI_{n-1}$.
\item Define $\pi_{n-1}:\cI_{n-1}\mapsto \cA_{n}$ and $\tau_{n-1}:\cF_{n-1}\to\cA_{n}$ to be the maps induced by the inclusion functions.
\item Define $\gamma_{n-1}:\cA_{n-1}\to \cA_{n}$ by $\gamma_{n-1}= \pi_{n-1}\circ i_{n-1} = \tau_{n-1}\circ f_{n-1}$.
\item For each $m< n$ define $\gamma_{mn} = \gamma_{n-1}\circ\ldots\circ\gamma_m$ (in particular, $\gamma_{m(m+1)} = \gamma_m$).
\item Define $\gamma_{nn}$ to be the identity map on $\cA_n$.
\end{itemize}

\begin{figure}
\[\xymatrix{ P\ar[r]^{i_0}\ar[d]_{f_0}\ar[dr]^{\gamma_0}  & \cI_0\ar[d]^{\pi_0}\\
\cF_0\ar[r]_{\tau_0} &  \cA_1\ar[r]^{i_1}\ar[d]_{f_1}\ar[dr]^{\gamma_1} & \cI_1\ar[d]^{\pi_1}\\
&  \cF_1\ar[r]_{\tau_1} & \cA_2\ar[r]^{i_2}\ar[d]_{f_2}\ar[dr]^{\gamma_2} & \cI_2\ar[d]^{\pi_2}\\
& &  \cF_2\ar[r]_{\tau_2} & \cA_3\ar@{..>}[dr] \\
& & & & \cA
}\] 
\caption{Building a free lattice}
\label{F:iter}
\end{figure}

The situation is presented as Figure \ref{F:iter}. $\cA$ is the object part of the colimit of the chain $\cA_0,\cA_1,\ldots$ as made precise in Theorem \ref{T:isColimit} later. For this we will need some technical results. The next lemma simply phrases the construction given at the start of this section in terms of a \emph{diagram}, in the categorical sense, and says that the maps in the resulting colimit are $(\cU_n,\cD_n)$-embeddings for all $n\in\omega$.

\begin{lemma}\label{L:muMorph}
Consider the ordinal $\omega$ as a category whose maps are induced by the order relation, and for each $m\leq n$ denote the map from $m$ to $n$ by $g_{mn}$. With $\cA_n$ etc. as defined at the start of this section, define a functor $F:\omega\to \Pos_e$ by $F(n) = \cA_n$ and $F(g_{nn})= id_{\cA_n}$, for all $n\in\omega$, and, for $m< n$, $F(g_{mn}) = \gamma_{m(n-1)}$. Let $(\mathcal A, \{\mu_n:n\in\omega\})$ be a colimit for $F$. Then $\mu_n$ is a $(\cU_n,\cD_n)$-embedding for all $n\in \omega$.
\end{lemma}
\begin{proof}
That $\mu_n$ is an order-embedding for all $n\in\omega$ follows from the fact that $\gamma_k$ is an order-embedding for all $k\in\omega$ (by Lemma \ref{L:gamma}). Now, given $k\in\omega$ and $T\in \cU_k$, it follows from Lemma \ref{L:gamma2} that $\gamma_{kn}(\bv T) = \bv \gamma_{kn}[T]$ for all $n\geq k$. Thus $\mu_k(\bv T) = \bv \mu_k[T]$, by Proposition \ref{P:coLim}(3).   In combination with the dual argument this gives us the result. 
\end{proof}

The next lemma describes $(\cU_n,\cD_n)$-morphisms from the $\cA_n$ posets into lattices in terms of meet- and join-preservation properties on the images of the $\pi_{n-1}$ and $\tau_{n-1}$ maps, for $n>0$. This will be used to show that the map induced by the universal property of colimits is a lattice homomorphism, and is thus the right kind of map for the universal property of free lattices.

\begin{lemma}\label{L:morphEquiv}
Let $n\in\omega\setminus\{0\}$, let $L$ be a lattice, and let $h:\cA_n\to L$. Then $h$ is a $(\cU_n,\cD_n)$-morphism if and only if $h$ is $\omega$-meet-preserving on $\tau_{n-1}[\cF_{n-1}]$, and $\omega$-join-preserving on $\pi_{n-1}[\cI_{n-1}]$.
\end{lemma}
\begin{proof}
Suppose $h$ is a $(\cU_n,\cD_n)$-morphism. Then, by definition of $\cU_n$, we have $h(\bv S) = \bv h[S]$ for all finite $S\subseteq \gamma_{n-1} [\cA_{n-1}]$. Let $Z\subseteq \cI_{n-1}$ with $|Z|<\omega$, and suppose $\bv Z$ is defined in $\cI_{n-1}$. By definition of $\cI_{n-1}$, each $y\in \cI_{n-1}$ is a finitely generated $\cU_{n-1}$-ideal. So, for each $y\in Z$ there is a finite $T_y\subseteq \cA_{n-1}$ with $y = \bv i_{n-1}[T_y]$ (in other words, $y$ is the smallest $\cU_{n-1}$-ideal containing $T_y$). Moreover, $\bigcup_{y\in Z} T_y$ is also finite, and $\bv i_{n-1}[\bigcup_{y\in Z} T_y]$ exists and is $\bv Z$. Now, $\pi_{n-1}$ is completely join-preserving, by Lemma \ref{L:gamma}, so 
\begin{align*}
h(\bv \pi_{n-1}[Z]) &= h\circ \pi_{n-1}(\bv Z) \\
&= h\circ \pi_{n-1}(\bv i_{n-1}[\bigcup_{Z} T_y]) \\
&= h(\bv \pi_{n-1}\circ i_{n-1}[\bigcup_{Z} T_y]) \\
&= h(\bv \gamma_{n-1}[\bigcup_{Z} T_y]) \\
&= \bv h\circ \gamma_{n-1}[\bigcup_{Z} T_y] \\
&= \bv h\circ \pi_{n-1}\circ i_{n-1}[\bigcup_{Z} T_y]\\
&= \bv h\circ \pi_{n-1}[Z].
\end{align*} 
Given finite $Z\subseteq \cF_{n-1}$ we also have $h(\bw \tau_{n-1}[Z]) = \bw h\circ \tau_{n-1}[Z]$ by a dual argument.

Conversely, suppose $h$ is $\omega$-join-preserving on $\pi_{n-1}[Y_{n-1}]$. Let $Z\subseteq \gamma_{n-1}[\cA_{n-1}]$ be finite. Then $Z = \pi_{n-1}\circ i_{n-1}[S]$ for some finite $S\subseteq \cA_{n-1}$, by definition of $\gamma_{n-1}$. Moreover, $\bv i_{n-1}[S]$ exists in $\cI_{n-1}$ (it's the smallest $\cU_{n-1}$-ideal containing $S$), and 
\[\bv Z = \bv \pi_{n-1}\circ i_{n-1}[S] = \pi_{n-1}(\bv i_{n-1}[S]),\]
 as $\pi_{n-1}$ is completely join-preserving. Thus $Z\subseteq \pi_{n-1}[\cI_{n-1}]$ and also $\bv Z \in \pi_{n-1}[\cI_{n-1}]$, and so, by the assumption that $h$ is $\omega$-join-preserving on $\pi_{n-1}[\cI_{n-1}]$, we have $h(\bv Z) = \bv h[Z]$. Thus $h$ is a $\cU_n$-morphism. That $h$ is a $\cD_n$-morphism whenever it is $\omega$-meet-preserving follows from a dual argument.  
\end{proof}

The next result says that $(\cU_0,\cD_0)$-morphisms from $P$ into lattices induce sequences of maps corresponding to a cocone. Thus the universal property of colimits produces a map that we shall show gives us what we want for the universal property of free lattices.
\begin{prop}\label{P:univA}
Let $L$ be a lattice, and let $h:P\to L$ be a $(\cU_0,\cD_0)$-morphism. Then there exists a sequence of maps $(h_0,h_1,\ldots)$, where $h_k:\cA_k\to L$ for each $k$, such that:
\begin{enumerate}
\item $h_k$ is a $(\cU_k,\cD_k)$-morphism for each $k$.
\item The appropriate part of the diagram in Figure \ref{F:buildF} commutes (ignoring the maps $\mu_0$ and $h^*$ for now).
\end{enumerate}
Moreover, the sequence $(h_0,h_1,\ldots)$ is unique with these properties.
\end{prop}
\begin{proof}
To show existence of such a sequence we prove by induction on $n$ that suitable subsequences $(h_0,\ldots,h_n)$ exist for all $n$. The base case $n=0$ is trivial, as we just set $h_0=h$, so given that $h=h_0,h_1,\ldots,h_{n}$ have been defined, we define $h_{n+1}:\cA_{n+1}\to L$ by
\[h_{n+1}(z) = \begin{cases} \bv h_{n}[I] \text{ when $z=I\in \cI_{n}$}\\  \bw h_n[F] \text{ when $z=F\in \cF_{n}$}\end{cases}\]
We must first check that $h_{n+1}$ is well defined. Note that $\cI_{n}$ and $\cF_{n}$ contain finitely generated $\cU_{n}$-ideals and $\cD_{n}$-filters, respectively. So, given $I\in\cI$, by definition there is a finite $S\subseteq \cA_{n}$ such that $I$ is the smallest $\cU_{n}$-ideal containing $S$. By Lemma \ref{L:S} we have $\bv h_n[I] = \bv h_n[S]$, with the latter join existing in $L$ as $S$ is finite. By this and a dual argument we see that the required joins and meets exist for the definition of $h_{n+1}$.  

In addition, we must check that $h_{n+1}(z)$ is well defined in the case where $I=z=F$ for some $I\in \cI_n$ and $F\in \cF_n$. For this, note first that, by definition, $I\leq F$ if and only if for all $p\in I$ and for all $q\in F$ we have $p\leq q$. From this it follows immediately that $\bv h_n[I]\leq \bw h_n[F]$. Similarly, $F\leq I$ if and only if there is $p\in F\cap I$, in which case $\bw h_n[F]\leq \bv h_n[I]$. Thus $h_{n+1}$ is well defined, and is also order preserving. 

Now, given $p\in \cA_n$, we have $h_{n+1}(p) = \bv h_n [p^\downarrow] = h_n(p)$, so the triangle involving $h_n$, $\gamma_n$ and $h_{n+1}$ commutes (recall that $\gamma_n(p) = p^\downarrow$).

Moreover, by Lemma \ref{L:morphEquiv}, to show that $h_{n+1}$ is a $(\cU_{n+1}, \cD_{n+1})$-morphism it is sufficient to show it is $\omega$-meet-preserving on $\tau_{n}[\cF_{n}]$, and $\omega$-join-preserving on $\pi_{n}[\cI_{n}]$. So let $Z\subseteq \pi_{n}[\cI_{n}]$, and suppose $\bv Z = I'$ in $\pi_{n}[\cI_{n}]$, so $I'$ is the smallest $\cU_n$-ideal containing $\bigcup Z$. As $\pi_n$ is completely join-preserving, by Lemma \ref{L:gamma}, we must have $\bv Z = I'$ in $\cA_{n+1}$ too. Now, for all $I\in Z$ we clearly have $h_{n+1}(I)\leq h_{n+1}(I')$. So, suppose $x\in L$ and that $h_{n+1}(I)\leq x$ for all $I\in Z$. Then, by Lemma \ref{L:ideals} and the inductive assumption that $h_n$ is a $(\cU_n,\cD_n)$-morphism, $h_n^{-1}(x^\downarrow)$ is a $\cU_n$-ideal, and also $I\subseteq h_n^{-1}(x^\downarrow)$ for all $I\in Z$, by definition of $h_{n+1}$. So $I'\subseteq h_n^{-1}(x^\downarrow)$, and thus $h_{n+1}(I')\leq x$. So $h_{n+1}$ is actually completely join-preserving on $\pi_{n}[\cI_{n}]$, and by a dual argument it is also completely meet-preserving on $\tau_{n}[\cF_{n}]$. Thus $h_{n+1}$ is a $(\cU_{n+1}, \cD_{n+1})$-morphism as claimed.

Finally, $h_{n+1}$ is unique with these properties, because given $I\in \cI_n$ generated by finite $S\subseteq \cA_n$, we have $I = \bv i_n[S]$ in $\cI_n$, and so $I = \bv \pi_n\circ i_n[S]=\bv \gamma_n[S]$ in $\cA_{n+1}$, as $\pi_n$ is completely join-preserving. If $h_{n+1}$ is to be a $(\cU_{n+1},\cD_{n+1})$-morphism then we must have $h_{n+1}(I) = \bv h_{n+1}\circ \gamma_n[S]$, since $\gamma_n[S]$ is a finite subset of $\gamma_n[\cA_n]$, and if the diagram is to commute we must have $h_{n+1}\circ \gamma_n[S] = h_n[S]$. So, appealing to Lemma \ref{L:S}, $h_{n+1}(I) = \bv h_n[S] = \bv h_n[I]$. By a dual argument we obtain that $h_{n+1}(F)$ can only be $\bw h_n[F]$, for $F\in\cF_n$, which completes the proof.

\end{proof}

\begin{figure}
\[\xymatrix{ P\ar@/^2pc/[rrrr]^{\mu_0}\ar[r]^{\gamma_0}\ar[dd]_{h} & \cA_1 \ar[ddl]_{h_1}\ar[r]^{\gamma_1} & \cA_2\ar@/^.3pc/[ddll]_{h_2}\ar[r]^{\gamma_2} & \cA_3\ar@/^.8pc/[ddlll]_{h_3}\ar@{.>}[r] & \cA\ar@/^1.5pc/[ddllll]^{h^*}\\
\\
L
}\] 
\caption{A sequence of $(\cU_n,\cD_n)$-morphisms}
\label{F:buildF}
\end{figure}

\begin{thm}\label{T:isColimit}
Consider the ordinal $\omega$ as a category, and for each $m\leq n\in\omega$ denote the map from $m$ to $n$ by $g_{mn}$. Define a functor $F:\omega\to \Pos_e$ so that $Fn = \cA_n$ for all $n\in\omega$. Define $F(g_{mm})$ to be the identity map for all $m\in\omega$, and define $F(g_{mn}) = \gamma_{m(n-1)}$ for all $m<n\in\omega$.  Let $(\mathcal A, \{\mu_n:n\in\omega\})$ be a colimit for $F$. Then $\mu_0:P\to \cA\cong \mu:P\to \FUD$. In other words, there is an isomorphism $\phi:\cA\to \FUD$ such that $\phi\circ \mu_0 = \mu$.  
\end{thm}
\begin{proof}

Let $L$ be a lattice and let $h:P\to L$ be a $(\cU, \cD)$-morphism. By Proposition \ref{P:univA} we obtain maps $h_1,h_2,\ldots$ such that the relevant parts of the diagram in Figure \ref{F:buildF} commutes. Thus, from the universal property of colimits we obtain a unique map $h^*:\cA\to L$ making the whole diagram commute. We must check that $\cA$ is a lattice, and that $h^*$ is the unique lattice homomorphism such that $h = h^*\circ \mu_0$.

So, let $x,y\in \cA$. Then there is $n$ such that $x,y\in\cA_n$ with $x\vee y$ and $x\wedge y$ defined in $\cA_n$, and with $\{x,y\}\in\cU_k=\cD_k$ for all $k\geq n$. As $\gamma_{nk}(x\vee y) = \gamma_{nk}(x)\vee \gamma_{nk}(y)$ for all $k\geq n$, by Lemma \ref{L:gamma2}, and similar for $\wedge$, it follows from Proposition \ref{P:coLim}(4) that $\cA$ is a lattice.

Moreover, since $x\vee y \in \cU_n$ and $h_n$ is a $(\cU_n,\cD_n)$-morphism (by Proposition \ref{P:univA}), it follows from the commutativity of the diagram that $h^*(x\vee y) = h^*(x)\vee h^*(y)$. Similar holds for $x\wedge y$, and thus $h^*$ is a lattice homomorphism.  

Finally, if $g:\cA\to L$ is a lattice homomorphism with $h = g\circ \mu_0$, then, for all $n$, the restriction of $g$ to $\cA_n$ is a $(\cU_n,\cD_n)$-morphism making the relevant part of the diagram in Figure \ref{F:buildF} commute, and so must be $h_n$, by Proposition \ref{P:univA}. It follows from the universal property of colimits that $g$ must be $h^*$. Thus $\mu_0:P\to \cA$ is the required free lattice (up to isomorphism), and we have the result.
\end{proof}

\section{Approximate lattice extensions}\label{S:approx}
The step by step construction of $\FUD$ from Section \ref{S:build} can be thought of as a sequence of increasingly good approximations. If $P$ is finite, then the free lattice $\FUD$ may not be. For example, the free lattice generated by a three element set is known to be infinite (see e.g. \cite[Theorem 1.28]{FJN95}). However, if $P$ is finite then $\cA_n$ will also be finite for each $n\in\omega$. Moreover, the map $\mu_n:\cA_n\to\cA$ is an order-embedding, and also preserves the meets and joins of all finite subsets of $\gamma_{n-1}[\cA_{n-1}]$. Thus, while each $\cA_n$ contains only a finite portion of the $(\cU,\cD)$-free lattice structure generated by $P$, there is a guarantee that much of what is contained in $\cA_n$ is correct. 

It follows that reasoning involving only terms of `bounded complexity', in a sense to be made precise in this section, can be done in $\cA_n$ for large enough $n$. For a simple example, it is obvious from this that the word problem for free lattices is solvable; Given terms $s$ and $t$ we can check whether $s\leq t$ `merely' by constructing $\cA_1,\cA_2,\ldots$ till we get to $\cA_n$ containing both $s$ and $t$, then checking whether $s\leq t$ in $\cA_n$. This is of course not a practical approach (see \cite[Chapter 9.8]{FJN95} for a discussion of algorithms for this problem).   

We can modify the result of Section \ref{S:build} to show that each stage $\cA_n$ also satisfies a kind of universal property. In this sense, these finite approximations to the free lattice are free objects themselves, albeit for a rather restrictive class.  We need some technical definitions to make this precise.
\begin{defn}
For each $2\leq n<\omega$ define $n$-ary operation symbols $\bv_n$ and $\bw_n$.
\end{defn}

\begin{defn}
Let $T$ be a set. Define \textbf{$T$-terms} recursively as follows: 
\begin{itemize}
\item If $t\in T$ then $t$ is a $T$-term.
\item If $2\leq n<\omega$ and $\phi_1,\ldots,\phi_n$ are $T$-terms, then $\bv_n(\phi_1,\ldots,\phi_n)$ and $\bw_n(\phi_1,\ldots,\phi_n)$ are $T$-terms.
\end{itemize}
We define the \textbf{complexity of $T$-terms} recursively as follows:
\begin{itemize}
\item If $t\in T$ then the complexity of $t$ is 0.
\item If $\phi_1,\ldots,\phi_n$ are $T$-terms with complexities $c_1,\ldots,c_n$ then $\bv_n(\phi_1,\ldots,\phi_n)$ and $\bw_n(\phi_1,\ldots,\phi_n)$ have complexity $\max(c_1,\ldots,c_n) + 1$.
\end{itemize}
\end{defn}

\begin{defn}
Let $Q$ be a poset, let $T\subseteq Q$, and let $\phi$ be a $T$-term. We define what it means for $q\in Q$ to \textbf{correspond to $\phi$} (or, equivalently, for $q$ to be a correspondent for $\phi$) as follows:
\begin{itemize}
\item If $\phi = t$ for some $t\in T$, then $q$ corresponds to $\phi$ if and only if $q= t$.
\item Suppose that $\phi = \bv_n(\phi_1,\ldots,\phi_n)$, and that $\phi_i$ is a $T$-term with correspondent $q_i$ for each $i\in\{1,\ldots,n\}$. Then $q$ corresponds to $\phi$ if and only if $q = \bv\{q_1,\ldots,q_n\}$.
\item Suppose that $\phi = \bw_n(\phi_1,\ldots,\phi_n)$, and that $\phi_i$ is a $T$-term with correspondent $q_i$ for each $i\in\{1,\ldots,n\}$. Then $q$ corresponds to $\phi$ if and only if $q = \bw\{q_1,\ldots,q_n\}$.
\end{itemize}
\end{defn}

Note that an easy inductive argument shows that a $T$-term has a unique correspondent, if it has one at all. However, an element $q\in Q$ may correspond to more than one $T$-term. 

\begin{defn}
Let $Q$ be a poset, let $k\in\omega\cup\{\omega\}$, and let $T\subseteq Q$. Then $Q$ is \textbf{$k$-complete relative to $T$} if, for all $k'<k$, every $T$-term of complexity $k'$ has a correspondent in $Q$.
\end{defn}

Note that every poset is trivially $1$-complete relative to every subset, as the terms with complexity 0 are just the elements of the subset.

\begin{defn}
Let $Q$ be a poset, and let $T\subseteq Q$. Given $q\in Q$, define $\rank_T(q)$ to be the least $n\in\omega$ such that $q$ corresponds to a $T$-term $\phi$ of complexity $n$, if such a $\phi$ exists, otherwise leave it undefined. 
\end{defn}

\begin{prop}
Let $P$ be a poset, let $1\leq n< \omega$, let $\gamma_{0n}:P\to\cA_n$ be as defined in Section \ref{S:build}, and let $\mu_0:P\to \cA$ be defined as in Theorem \ref{T:isColimit}. Then:
\begin{enumerate}
\item If $q\in \cA_n$ then $\rank_{\gamma_{0n}[P]}(q)\leq n$.  
\item $\cA_n$ is $(n+1)$-complete relative to $\gamma_{0n}[P]$.
\item If $q\in \cA$ then $\rank_{\mu_0[P]}(q)$ is finite.
\item $\cA$ is $\omega$-complete relative to $\mu_0[P]$.
\item Let $q\in \cA$ and let $n\in\omega$. Then $\rank_{\mu_0[P]}(q) = n$ if and only if $n$ is the smallest number such that there is $q'\in\cA_n$ with $\mu_n(q') = q$.
\end{enumerate}
\end{prop}
\begin{proof}
Parts (1) and (2) can be proved by easy inductions on $n$. Parts (3) and (4) then follow from the fact that, for all $n\in\omega$, the map $\mu_n:\cA_n\to\cA$ is a $(\cU_n,\cD_n)$-embedding (by Lemma \ref{L:muMorph}), and for every $q\in\cA$ we have $q = \mu_n(q')$ for some $n\in\omega$ and $q'\in\cA_n$. 

Part (5) also follows by an induction argument. The case where $n=0$ is trivial, so suppose $n>0$ and that the claim holds for all $m<n$, and let $q\in \cA$. Suppose first that $\rank_{\mu_0[P]}(q)=n$, and let $\phi$ be a $\mu_0[P]$-term of complexity $n$ to which $q$ corresponds. Suppose without loss of generality that $\phi = \bv_k(\phi_1,\ldots,\phi_k)$ for some $\mu_0[P]$-terms $\phi_1,\ldots,\phi_k$, each of which has complexity of at most $n-1$. For each $i\in\{1,\ldots,k\}$ let $q_i\in\cA$ be the correspondent of $\phi_i$. Then, for all $i\in\{1,\ldots,k\}$ we have $\rank_{\mu_0[P]}(q_i) < n$.  

Now, by the inductive hypothesis, for each $i\in\{1,\ldots,k\}$ there is $n_i< n$ and $q'_i\in\cA_{n_i}$, with $q_i = \mu_{n_i}(q'_i)$. Let $n' = \max(n_1,\ldots,n_k)$. As $q$ corresponds to $\phi$, there must be $q'\in \cA_{n'+1}$ such that $\mu_{n'+1}(q') = q$. Moreover, if there were $n''<n$ and $q''\in\cA_{n''}$ such that $\mu_{n''}(q'') = q$ then, also by the inductive hypothesis, we would have $\rank_{\mu_0[P]}(q) < n$, contradicting the assumption that $\rank_{\mu_0[P]}(q) = n$. It follows that $n'+1 = n$, and that $n$ is indeed the smallest number such that there is $q'\in\cA_n$ with $\mu_n(q') = q$.

For the converse, suppose $n$ is the smallest number such that there is $q'\in\cA_n$ with $\mu_n(q') = q$. Then there are $q'_1,\ldots,q'_k\in\cA_{n-1}$ such that either $q' = \bv\{\gamma_{n-1}(q'_1),\ldots,\gamma_{n-1}(q'_k)\}$, or $q' = \bw\{\gamma_{n-1}(q'_1),\ldots,\gamma_{n-1}(q'_k)\}$. Now, for each $i\in\{1,\ldots,k\}$ let $q_i = \mu_{n-1}(q'_i)$, and let $\phi_i$ correspond to $q_i$ and have minimal complexity. Suppose without loss of generality that $q' = \bv\{\gamma_{n-1}(q'_1),\ldots,\gamma_{n-1}(q'_k)\}$. Then $q$ corresponds to $\bv_k(\phi_1,\ldots,\phi_k)$, and, as $\rank_{\mu_0[P]}(q_i) < n$ for all $i\in\{1,\ldots,k\}$, it follows that $\rank_{\mu_0[P]}(q) \leq n$. Moreover, if $\rank_{\mu_0[P]}(q) < n$ then, by the inductive hypothesis, $n$ could not be minimal as assumed. It follows that, if $n$ is the smallest number such that there is $q'\in\cA_n$ with $\mu_n(q') = q$, then $\rank_{\mu_0[P]}(q) = n$.          
\end{proof}

\begin{thm}\label{T:rel-free}
Let $P$ and $Q$ be posets, let $n\in \omega$, and let $\cA_n$ be as defined in Section \ref{S:build}. Let $f:P\to Q$ be a $(\cU,\cD)$-morphism, and suppose $Q$ is $(n+1)$-complete relative to $f[P]$. Then there is a unique $(\cU_n,\cD_n)$-morphism $f^*:\cA_n\to Q$ such that $f^*\circ \gamma_{0n} = f$. 
\end{thm}
\begin{proof}
First, if $n=0$ then $\cA_n = P$ and the result is trivial. Suppose then that $n>0$, and that the claim holds for all $k<n$. The argument now is essentially that of Proposition \ref{P:univA}. The only difference is that, as $Q$ is not a lattice, it is not immediately obvious that $Q$ has the required joins and meets. However, a little reflection reveals that the satisfaction of these conditions to a degree sufficient to prove the claimed result follows from the fact that $Q$ is $(n+1)$-complete relative to $f[P]$.  
\end{proof}

Theorem \ref{T:rel-free} says, in a sense, that $\cA_n$ is the free poset generated by $P$ (while preserving certain bounds) that has lattice structure up to a certain level of complexity, if using elements of $P$ as a base. 

In the case where $P$ is an antichain, there is a well known `canonical form' theorem, which, in our notation, produces for each $q\in\cA$ a $\mu_0[P]$-term corresponding to $q$ that is minimal with respect to a certain measure of complexity, and this term is `unique up to commutativity' (see e.g. \cite[Theorem 1.17]{FJN95}). Unfortunately, this theorem does not hold for posets in general, as we illustrate in Example \ref{E:canon}.

\begin{ex}\label{E:canon}
Let $P$ be the poset in Figure \ref{F:canon}, let $\cU$ and $\cD$ contain, respectively, all joins and meets that are defined in $P$, and consider the element $a\vee b$. This is not defined in $P$, but is defined in $\cA_1$, and is, in the construction of $\cA_1$ using $\cU$-ideals and $\cD$-filters, the smallest $\cU$-ideal containing $\{a,b\}$. Inspection reveals this is the whole of $P$. Now, the smallest $\cU$-ideal containing $\{x,y\}$ is also the whole of $P$, and thus $a\vee b = x\vee y$. But $\{a,b\}$ and $\{x,y\}$ are disjoint, and there is no natural reason to choose one over the other as the basis for a canonical term for the element corresponding to the join in $\cA_1$. Since $\cA_1$ correctly represents the joins of elements of $P$ in the colimit $\cA$, this argument reveals that a canonical form theorem such as exists for free lattices over sets does not exist in this more general setting.
\end{ex}

\begin{figure}
\[\xymatrix{& \bullet_a\ar@{-}[ddl]\ar@{-}[ddd] & & \bullet_b & & &  \bullet_x & & \bullet_y\ar@{-}[ddd]\ar@{-}[ddr] \\
& & \bullet\ar@{-}[dll]\ar@{-}[urrrr] & & & & & \bullet\ar@{-}[drr]\ar@{-}[ullll]  \\
\bullet & & \bullet\ar@{-}[dl]\ar@{-}[uurrrrrr] & & & & & \bullet\ar@{-}[dr]\ar@{-}[uullllll] & & \bullet \\
& \bullet & & \bullet\ar@{-}[ul]\ar@{-}[uul]\ar@{-}[uuu] & & &  \bullet\ar@{-}[uur]\ar@{-}[ur]\ar@{-}[uuu] & & \bullet
}\] 
\caption{The bat signal poset}
\label{F:canon}
\end{figure}

%\paragraph{Data availability} This paper contains no data.
%\bibliography{../../../../../references}{}
\bibliographystyle{abbrv}

\end{document}